\documentclass{amsart}
\usepackage{amsfonts}

\newtheorem{theorem}{Theorem}
\theoremstyle{plain}

\newtheorem{corollary}{Corollary}

\newtheorem{lemma}{Lemma}

\numberwithin{equation}{section}
\input{tcilatex}

\begin{document}
\title{Inequalities via $n$-times differentiable $quasi-$convex functions}
\author{Merve Avci Ardic}
\address{Adiyaman University, Faculty of Science and Arts, Department of
Mathematics, 02040, Adiyaman}
\email{mavci@adiyaman.edu.tr}

\begin{abstract}
In this paper, we establish some integral ineuqalities for $n-$ times
differentiable $quasi-$convex functions.
\end{abstract}

\maketitle

\section{introduction}

A function $\ f:[a,b]\rightarrow 
\mathbb{R}
$ is said to be $quasi-$convex function on $[a,b]$ if the inequality%
\begin{equation*}
f(\alpha x+(1-\alpha )y)\leq \max \left\{ f(x),f(y)\right\} 
\end{equation*}%
holds for all $x,y\in \lbrack a,b]$ and $\alpha \in \lbrack 0,1].$For some
results about $quasi-$convexity see \cite{16}-\cite{20}.

Hermite-Hadamard inequality is defined below:

Let $f:I\subset 
\mathbb{R}
\rightarrow 
\mathbb{R}
$ be a convex function on an interval $I$ and $a,b\in I$ with $a<b.$ Then 
\begin{equation*}
f\left( \frac{a+b}{2}\right) \leq \frac{1}{b-a}\int_{a}^{b}f(x)dx\leq \frac{%
f(a)+f(b)}{2}.
\end{equation*}
\ \ \ \ \ \ \ \ \ For some inequalities, generalizations and applications
concerning Hermite-Hadamard inequality see \cite{1}-\cite{6} and \cite{10}, 
\cite{11}.

Recently, in the literature there are so many manuscripts about $n-$times
differentiable functions on several kinds of convexities. In references \cite%
{6}-\cite{12}, readers can find some results about this issue.

Unlike the above-referenced studies, this paper aims to establish integral
inequalities for $n-$times differentiable $quasi-$convex functions.

To prove our main results, we use the following lemmas.

\begin{lemma}
\label{lem 1.1} \cite{21} Let $f:[a,b]\rightarrow 
\mathbb{R}
$ be a mapping such that the derivative $f^{(n-1)}$ $(n\geq 1)$ is
absolutely continuous on $[a,b]$. Then for any $x\in \lbrack a,b]$ one has
the equality:%
\begin{eqnarray*}
\int_{a}^{b}f(t)dt &=&\sum_{k=0}^{n-1}\frac{1}{\left( k+1\right) !}\left[
\left( x-a\right) ^{k+1}f^{(k)}(a)+\left( -1\right) ^{k}\left( b-x\right)
^{k+1}f^{(k)}(b)\right]  \\
&&+\frac{1}{n!}\int_{a}^{b}\left( x-t\right) ^{n}f^{(n)}(t)dt.
\end{eqnarray*}
\end{lemma}

\begin{lemma}
\label{lem 1.2} \cite{10.} Let $f:I^{\circ }\subseteq 
\mathbb{R}
\rightarrow 
\mathbb{R}
,$ $a,b\in I^{\circ }$ with $a<b,$ $f^{(n)}$ exists on $I^{\circ }$ and $%
f^{(n)}\in L(a,b)$ for $n\geq 1.$ We have the following identity:\ 
\begin{eqnarray*}
&&\frac{f(a)+f(b)}{2}-\frac{1}{b-a}\int_{a}^{b}f(x)dx-\sum_{k=2}^{n-1}\frac{%
\left( k-1\right) \left( b-a\right) ^{k}}{2(k+1)!}f^{(k)}(a) \\
&=&\frac{\left( b-a\right) ^{n}}{2n!}\int_{0}^{1}t^{n-1}(n-2t)f^{(n)}(ta+%
\left( 1-t\right) b)dt.
\end{eqnarray*}
\end{lemma}

\section{inequalities for $n-$times differentiable $quasi-$convex functions
via Lemma \protect\ref{lem 1.1}}

The first result is the following theorem. 

\begin{theorem}
\label{teo 2.1} Let $f:[a,b]\rightarrow 
\mathbb{R}
$ be a mapping such that the derivative $f^{(n-1)}$ $(n\geq 1)$ is
absolutely continuous on $[a,b]$. If $\left\vert f^{(n)}\right\vert $ is $%
quasi-$convex on $[a,b]$ following inequality holds%
\begin{eqnarray*}
&&\left\vert \int_{a}^{b}f(t)dt-\sum_{k=0}^{n-1}\frac{1}{\left( k+1\right) !}%
\left[ \left( x-a\right) ^{k+1}f^{(k)}(a)+\left( -1\right) ^{k}\left(
b-x\right) ^{k+1}f^{(k)}(b)\right] \right\vert  \\
&\leq &\frac{1}{n!}\left\{ \left[ \max \left\{ \left\vert
f^{(n)}(a)\right\vert ,\left\vert f^{(n)}(x)\right\vert \right\} \right] 
\frac{\left( x-a\right) ^{n+1}}{n+1}\right.  \\
&&+\left. \left[ \max \left\{ \left\vert f^{(n)}(x)\right\vert ,\left\vert
f^{(n)}(b)\right\vert \right\} \right] \frac{\left( b-x\right) ^{n+1}}{n+1}%
\right\} 
\end{eqnarray*}%
for all $x\in \lbrack a,b]$.

\begin{proof}
From Lemma \ref{lem 1.1}, property of the modulus and $quasi$-convexity of $%
\left\vert f^{(n)}\right\vert ,$ we can write%
\begin{eqnarray*}
&&\left\vert \int_{a}^{b}f(t)dt-\sum_{k=0}^{n-1}\frac{1}{\left( k+1\right) !}%
\left[ \left( x-a\right) ^{k+1}f^{(k)}(a)+\left( -1\right) ^{k}\left(
b-x\right) ^{k+1}f^{(k)}(b)\right] \right\vert  \\
&\leq &\frac{1}{n!}\int_{a}^{b}\left\vert x-t\right\vert ^{n}\left\vert
f^{(n)}(t)\right\vert dt \\
&=&\frac{1}{n!}\left\{ \int_{a}^{x}\left( x-t\right) ^{n}\left\vert
f^{(n)}(t)\right\vert dt+\int_{x}^{b}\left( t-x\right) ^{n}\left\vert
f^{(n)}(t)\right\vert dt\right\}  \\
&=&\frac{1}{n!}\left\{ \int_{a}^{x}\left( x-t\right) ^{n}\left\vert
f^{(n)}\left( \frac{t-a}{x-a}x+\frac{x-t}{x-a}a\right) \right\vert dt\right. 
\\
&&\left. +\int_{x}^{b}\left( t-x\right) ^{n}\left\vert f^{(n)}\left( \frac{%
t-x}{b-x}b+\frac{b-t}{b-x}x\right) \right\vert dt\right\}  \\
&\leq &\frac{1}{n!}\left\{ \int_{a}^{x}\left( x-t\right) ^{n}\left[ \max
\left\{ \left\vert f^{(n)}(a)\right\vert ,\left\vert f^{(n)}(x)\right\vert
\right\} \right] \right.  \\
&&\left. +\int_{x}^{b}\left( t-x\right) ^{n}\left[ \max \left\{ \left\vert
f^{(n)}(x)\right\vert ,\left\vert f^{(n)}(b)\right\vert \right\} \right]
\right\} .
\end{eqnarray*}%
If we use the fact that%
\begin{equation*}
\int_{a}^{x}\left( x-t\right) ^{n}dt=\frac{\left( x-a\right) ^{n+1}}{n+1}
\end{equation*}%
and%
\begin{equation*}
\int_{x}^{b}\left( t-x\right) ^{n}dt=\frac{(b-x)^{n+1}}{(n+1)}
\end{equation*}%
we get the desired result.
\end{proof}
\end{theorem}

\begin{corollary}
\label{co 1.1} In Theorem \ref{teo 2.1}, if $x=\frac{a+b}{2},$ following
inequality holds:%
\begin{eqnarray*}
&&\left\vert \int_{a}^{b}f(t)dt-\sum_{k=0}^{n-1}\frac{1}{\left( k+1\right) !}%
\left( \frac{b-a}{2}\right) ^{k+1}\left[ f^{(k)}(a)+\left( -1\right)
^{k}f^{(k)}(b)\right] \right\vert  \\
&\leq &\frac{(b-a)^{n+1}}{2^{n+1}(n+1)!}\left\{ \left[ \max \left\{
\left\vert f^{(n)}(a)\right\vert ,\left\vert f^{(n)}(x)\right\vert \right\} %
\right] +\left[ \max \left\{ \left\vert f^{(n)}(x)\right\vert ,\left\vert
f^{(n)}(b)\right\vert \right\} \right] \right\} .
\end{eqnarray*}
\end{corollary}

\begin{theorem}
\label{teo 2.2} Let $f:[a,b]\rightarrow 
\mathbb{R}
$ be a mapping such that the derivative $f^{(n-1)}$ $(n\geq 1)$ is
absolutely continuous on $[a,b]$. If $\left\vert f^{(n)}\right\vert ^{q}$ is 
$quasi-$convex on $[a,b],$ following inequality holds for all $x\in \lbrack
a,b]$%
\begin{eqnarray*}
&&\left\vert \int_{a}^{b}f(t)dt-\sum_{k=0}^{n-1}\frac{1}{\left( k+1\right) !}%
\left[ \left( x-a\right) ^{k+1}f^{(k)}(a)+\left( -1\right) ^{k}\left(
b-x\right) ^{k+1}f^{(k)}(b)\right] \right\vert  \\
&\leq &\frac{\left( b-a\right) ^{1/q}}{n!}\left( \frac{\left( x-a\right)
^{np+1}+\left( b-x\right) ^{np+1}}{np+1}\right) ^{1/p}\left[ \max \left\{
\left\vert f^{(n)}(a)\right\vert ^{q},\left\vert f^{(n)}(b)\right\vert
^{q}\right\} \right] ^{\frac{1}{q}}
\end{eqnarray*}%
where $p>1$ and $\frac{1}{p}+\frac{1}{q}=1.$
\end{theorem}

\begin{proof}
From Lemma \ref{lem 1.1}, property of the modulus, well-known H\"{o}lder
integral inequality and $quasi-$convexity of $\left\vert f^{(n)}\right\vert
^{q},$ we can write%
\begin{eqnarray*}
&&\left\vert \int_{a}^{b}f(t)dt-\sum_{k=0}^{n-1}\frac{1}{\left( k+1\right) !}%
\left[ \left( x-a\right) ^{k+1}f^{(k)}(a)+\left( -1\right) ^{k}\left(
b-x\right) ^{k+1}f^{(k)}(b)\right] \right\vert  \\
&\leq &\frac{1}{n!}\left( \int_{a}^{b}\left\vert x-t\right\vert
^{np}dt\right) ^{1/p}\left( \int_{a}^{b}\left\vert f^{(n)}(t)\right\vert
^{q}dt\right) ^{1/q} \\
&=&\frac{1}{n!}\left( \int_{a}^{x}\left( x-t\right)
^{np}dt+\int_{x}^{b}\left( t-x\right) ^{np}dt\right) ^{1/p}\left(
\int_{a}^{b}\left\vert f^{(n)}\left( \frac{b-t}{b-a}a+\frac{t-a}{b-a}%
b\right) \right\vert ^{q}dt\right) ^{1/q} \\
&\leq &\frac{1}{n!}\left( \frac{\left( x-a\right) ^{np+1}+\left( b-x\right)
^{np+1}}{np+1}\right) ^{1/p}\left( \int_{a}^{b}\left[ \frac{b-t}{b-a}%
\left\vert f^{(n)}(a)\right\vert ^{q}+\frac{t-a}{b-a}\left\vert
f^{(n)}(b)\right\vert ^{q}\right] dt\right) ^{1/q} \\
&=&\frac{\left( b-a\right) ^{1/q}}{n!}\left( \frac{\left( x-a\right)
^{np+1}+\left( b-x\right) ^{np+1}}{np+1}\right) ^{1/p}\left[ \max \left\{
\left\vert f^{(n)}(a)\right\vert ^{q},\left\vert f^{(n)}(b)\right\vert
^{q}\right\} \right] ^{\frac{1}{q}}.
\end{eqnarray*}%
The proof is completed.
\end{proof}

\begin{theorem}
\label{teo 2.3} Let $f:[a,b]\rightarrow 
\mathbb{R}
$ be a mapping such that the derivative $f^{(n-1)}$ $(n\geq 1)$ is
absolutely continuous on $[a,b]$. If $\left\vert f^{(n)}\right\vert ^{q}$ is 
$quasi-$convex on $[a,b],$ following inequality holds for all $x\in \lbrack
a,b]$%
\begin{eqnarray*}
&&\left\vert \int_{a}^{b}f(t)dt-\sum_{k=0}^{n-1}\frac{1}{\left( k+1\right) !}%
\left[ \left( x-a\right) ^{k+1}f^{(k)}(a)+\left( -1\right) ^{k}\left(
b-x\right) ^{k+1}f^{(k)}(b)\right] \right\vert  \\
&\leq &\frac{1}{n!}\left( \frac{q-1}{nq-p+q-1}\right) ^{\frac{1}{p}} \\
&&\times \left\{ \left( \left( x-a\right) ^{\left( nq+q-p-1\right)
/(q-1)}\right) ^{\frac{1}{p}}\left( \frac{\left( x-a\right) ^{p+1}}{p+1}%
\right) ^{\frac{1}{q}}\left[ \max \left\{ \left\vert f^{(n)}(a)\right\vert
^{q},\left\vert f^{(n)}(x)\right\vert ^{q}\right\} \right] ^{\frac{1}{q}%
}\right.  \\
&&\left. +\left( \left( b-x\right) ^{\left( nq+q-p-1\right) /(q-1)}\right) ^{%
\frac{1}{p}}\left( \frac{\left( b-x\right) ^{p+1}}{p+1}\right) ^{\frac{1}{q}}%
\left[ \max \left\{ \left\vert f^{(n)}(x)\right\vert ^{q},\left\vert
f^{(n)}(b)\right\vert ^{q}\right\} \right] ^{\frac{1}{q}}\right\} 
\end{eqnarray*}%
where $p>1$ and $\frac{1}{p}+\frac{1}{q}=1.$
\end{theorem}

\begin{proof}
From Lemma \ref{lem 1.1}, property of the modulus, well-known H\"{o}lder
integral inequality and $quasi-$convexity of  $\left\vert f^{(n)}\right\vert
^{q}$ we can write%
\begin{eqnarray*}
&&\left\vert \int_{a}^{b}f(t)dt-\sum_{k=0}^{n-1}\frac{1}{\left( k+1\right) !}%
\left[ \left( x-a\right) ^{k+1}f^{(k)}(a)+\left( -1\right) ^{k}\left(
b-x\right) ^{k+1}f^{(k)}(b)\right] \right\vert  \\
&\leq &\frac{1}{n!}\int_{a}^{b}\left\vert x-t\right\vert ^{n}\left\vert
f^{(n)}(t)\right\vert dt \\
&=&\frac{1}{n!}\left\{ \int_{a}^{x}\left( x-t\right) ^{n}\left\vert
f^{(n)}(t)\right\vert dt+\int_{x}^{b}\left( t-x\right) ^{n}\left\vert
f^{(n)}(t)\right\vert dt\right\}  \\
&\leq &\frac{1}{n!}\left\{ \left( \int_{a}^{x}\left( x-t\right)
^{(nq-p)/(q-1)}dt\right) ^{\frac{1}{p}}\left( \int_{a}^{x}\left( x-t\right)
^{p}\left\vert f^{(n)}(t)\right\vert ^{q}dt\right) ^{\frac{1}{q}}\right.  \\
&&\left. +\left( \int_{x}^{b}\left( t-x\right) ^{(nq-p)/(q-1)}dt\right) ^{%
\frac{1}{p}}\left( \int_{x}^{b}\left( t-x\right) ^{p}\left\vert
f^{(n)}(t)\right\vert ^{q}dt\right) ^{\frac{1}{q}}\right\}  \\
&=&\frac{1}{n!}\left\{ \left( \int_{a}^{x}\left( x-t\right)
^{(nq-p)/(q-1)}dt\right) ^{\frac{1}{p}}\left( \int_{a}^{x}\left( x-t\right)
^{p}\left\vert f^{(n)}\left( \frac{t-a}{x-a}x+\frac{x-t}{x-a}a\right)
\right\vert ^{q}dt\right) ^{\frac{1}{q}}\right.  \\
&&\left. +\left( \int_{x}^{b}\left( t-x\right) ^{(nq-p)/(q-1)}dt\right) ^{%
\frac{1}{p}}\left( \int_{x}^{b}\left( t-x\right) ^{p}\left\vert
f^{(n)}\left( \frac{t-x}{b-x}b+\frac{b-t}{b-x}x\right) \right\vert
^{q}dt\right) ^{\frac{1}{q}}\right\}  \\
&\leq &\frac{1}{n!}\left\{ \left( \int_{a}^{x}\left( x-t\right)
^{(nq-p)/(q-1)}dt\right) ^{\frac{1}{p}}\left( \int_{a}^{x}\left( x-t\right)
^{p}\left[ \max \left\{ \left\vert f^{(n)}(a)\right\vert ^{q},\left\vert
f^{(n)}(x)\right\vert ^{q}\right\} \right] dt\right) ^{\frac{1}{q}}\right. 
\\
&&\left. +\left( \int_{x}^{b}\left( t-x\right) ^{(nq-p)/(q-1)}dt\right) ^{%
\frac{1}{p}}\left( \int_{x}^{b}\left( t-x\right) ^{p}\left[ \max \left\{
\left\vert f^{(n)}(x)\right\vert ^{q},\left\vert f^{(n)}(b)\right\vert
^{q}\right\} \right] dt\right) ^{\frac{1}{q}}\right\} .
\end{eqnarray*}%
If we use the fact that 
\begin{equation*}
\int_{a}^{x}\left( x-t\right) ^{(nq-p)/(q-1)}dt=\frac{q-1}{nq-p+q-1}\left(
x-a\right) ^{\left( nq+q-p-1\right) /(q-1)},
\end{equation*}%
\begin{equation*}
\int_{x}^{b}\left( t-x\right) ^{(nq-p)/(q-1)}dt=\frac{q-1}{nq-p+q-1}\left(
b-x\right) ^{\left( nq+q-p-1\right) /(q-1)},
\end{equation*}%
\begin{equation*}
\int_{a}^{x}\left( x-t\right) ^{p}dt=\frac{\left( x-a\right) ^{p+1}}{p+1}
\end{equation*}%
and%
\begin{equation*}
\int_{x}^{b}\left( t-x\right) ^{p}dt=\frac{\left( b-x\right) ^{p+1}}{p+1}
\end{equation*}%
in belove we complete the proof.
\end{proof}

\section{inequalities for $n-$times differentiable $quasi-$convex functions
via Lemma \protect\ref{lem 1.2}}

\begin{theorem}
\label{teo 3.1} Let $f:[a,b]\rightarrow 
\mathbb{R}
$ be an $n-$ times differentiable function $.$ If $\left\vert
f^{(n)}\right\vert ^{q}$ is $quasi-$convex on $[a,b]$ following inequality
holds\ 
\begin{eqnarray*}
&&\left\vert \frac{f(a)+f(b)}{2}-\frac{1}{b-a}\int_{a}^{b}f(x)dx-%
\sum_{k=2}^{n-1}\frac{\left( k-1\right) \left( b-a\right) ^{k}}{2(k+1)!}%
f^{(k)}(a)\right\vert  \\
&=&\frac{\left( b-a\right) ^{n}\left( n-1\right) }{2(n+1)!}\left[ \max
\left\{ \left\vert f^{(n)}(a)\right\vert ^{q},\left\vert
f^{(n)}(b)\right\vert ^{q}\right\} \right] ^{\frac{1}{q}}
\end{eqnarray*}%
for $q\geq 1$ and $\ n\geq 2.$

\begin{proof}
From Lemma \ref{lem 1.2}, property of the modulus, well-known power-mean
integral inequality and $quasi-$convexity of $\left\vert f^{(n)}\right\vert
^{q}$ we can write%
\begin{eqnarray*}
&&\left\vert \frac{f(a)+f(b)}{2}-\frac{1}{b-a}\int_{a}^{b}f(x)dx-%
\sum_{k=2}^{n-1}\frac{\left( k-1\right) \left( b-a\right) ^{k}}{2(k+1)!}%
f^{(k)}(a)\right\vert  \\
&\leq &\frac{\left( b-a\right) ^{n}}{2n!}\int_{0}^{1}t^{n-1}(n-2t)\left\vert
f^{(n)}(ta+\left( 1-t\right) b)\right\vert dt \\
&\leq &\frac{\left( b-a\right) ^{n}}{2n!}\left(
\int_{0}^{1}t^{n-1}(n-2t)dt\right) ^{1-\frac{1}{q}}\left(
\int_{0}^{1}t^{n-1}(n-2t)\left\vert f^{(n)}(ta+\left( 1-t\right)
b)\right\vert ^{q}dt\right) ^{\frac{1}{q}} \\
&\leq &\frac{\left( b-a\right) ^{n}}{2n!}\left( \frac{n-1}{n+1}\right) ^{1-%
\frac{1}{q}}\left( \int_{0}^{1}t^{n-1}(n-2t)\left[ \max \left\{ \left\vert
f^{(n)}(a)\right\vert ^{q},\left\vert f^{(n)}(b)\right\vert ^{q}\right\} %
\right] dt\right) ^{\frac{1}{q}} \\
&=&\frac{\left( b-a\right) ^{n}(n-1)}{2(n+1)!}\left[ \max \left\{ \left\vert
f^{(n)}(a)\right\vert ^{q},\left\vert f^{(n)}(b)\right\vert ^{q}\right\} %
\right] ^{\frac{1}{q}}
\end{eqnarray*}%
which completes the proof.
\end{proof}
\end{theorem}

\begin{corollary}
\label{co 3.1} In Theorem \ref{teo 3.1}, if we choose $q=1$ we obtain%
\begin{eqnarray*}
&&\left\vert \frac{f(a)+f(b)}{2}-\frac{1}{b-a}\int_{a}^{b}f(x)dx-%
\sum_{k=2}^{n-1}\frac{\left( k-1\right) \left( b-a\right) ^{k}}{2(k+1)!}%
f^{(k)}(a)\right\vert  \\
&=&\frac{\left( b-a\right) ^{n}\left( n-1\right) }{2(n+1)!}\left[ \max
\left\{ \left\vert f^{(n)}(a)\right\vert ,\left\vert f^{(n)}(b)\right\vert
\right\} \right] .
\end{eqnarray*}
\end{corollary}

\begin{theorem}
\label{teo 3.2} Let $f:[a,b]\rightarrow 
\mathbb{R}
$ be an $n-$ times differentiable function $.$ If $\left\vert
f^{(n)}\right\vert ^{q}$ is $quasi-$convex on $[a,b]$ following inequality
holds%
\begin{eqnarray*}
&&\left\vert \frac{f(a)+f(b)}{2}-\frac{1}{b-a}\int_{a}^{b}f(x)dx-%
\sum_{k=2}^{n-1}\frac{\left( k-1\right) \left( b-a\right) ^{k}}{2(k+1)!}%
f^{(k)}(a)\right\vert  \\
&\leq &\frac{\left( b-a\right) ^{n}}{2^{1+\frac{1}{q}}n!}\left( \frac{1}{%
np-p+1}\right) ^{\frac{1}{p}}\left( \frac{n^{q+1}-\left( n-2\right) ^{q+1}}{%
q+1}\right) ^{\frac{1}{q}} \\
&&\times \left[ \max \left\{ \left\vert f^{(n)}(a)\right\vert
^{q},\left\vert f^{(n)}(b)\right\vert ^{q}\right\} \right] ^{\frac{1}{q}}
\end{eqnarray*}%
for $n\geq 2,$ $q>1$ and $\frac{1}{p}+\frac{1}{q}=1.$
\end{theorem}

\begin{proof}
From Lemma \ref{lem 1.2}, property of the modulus, well-known H\"{o}lder
integral inequality and $quasi-$convexity of $\left\vert f^{(n)}\right\vert
^{q}$ we can write%
\begin{eqnarray*}
&&\left\vert \frac{f(a)+f(b)}{2}-\frac{1}{b-a}\int_{a}^{b}f(x)dx-%
\sum_{k=2}^{n-1}\frac{\left( k-1\right) \left( b-a\right) ^{k}}{2(k+1)!}%
f^{(k)}(a)\right\vert  \\
&\leq &\frac{\left( b-a\right) ^{n}}{2n!}\int_{0}^{1}t^{n-1}(n-2t)\left\vert
f^{(n)}(ta+\left( 1-t\right) b)\right\vert dt \\
&\leq &\frac{\left( b-a\right) ^{n}}{2n!}\left(
\int_{0}^{1}t^{(n-1)p}dt\right) ^{\frac{1}{p}}\left(
\int_{0}^{1}(n-2t)^{q}\left\vert f^{(n)}(ta+\left( 1-t\right) b)\right\vert
^{q}dt\right) ^{\frac{1}{q}} \\
&\leq &\frac{\left( b-a\right) ^{n}}{2n!}\left( \frac{1}{np-p+1}\right) ^{%
\frac{1}{p}}\left( \int_{0}^{1}(n-2t)^{q}\left[ \max \left\{ \left\vert
f^{(n)}(a)\right\vert ^{q},\left\vert f^{(n)}(b)\right\vert ^{q}\right\} %
\right] dt\right) ^{\frac{1}{q}} \\
&=&\frac{\left( b-a\right) ^{n}}{2^{1+\frac{1}{q}}n!}\left( \frac{1}{np-p+1}%
\right) ^{\frac{1}{p}}\left( \frac{n^{q+1}-\left( n-2\right) ^{q+1}}{q+1}%
\right) ^{\frac{1}{q}} \\
&&\times \left[ \max \left\{ \left\vert f^{(n)}(a)\right\vert
^{q},\left\vert f^{(n)}(b)\right\vert ^{q}\right\} \right] ^{\frac{1}{q}}.
\end{eqnarray*}%
The proof is completed.
\end{proof}

\end{document}